\documentclass[11pt]{amsart}
\usepackage[margin=1in]{geometry}

\usepackage{amsmath}
\usepackage[linkcolor=blue,colorlinks=true,citecolor=blue]{hyperref}
\usepackage{graphicx}
\usepackage{subcaption}
\usepackage{cleveref}
\usepackage{xcolor}

\usepackage{thmtools}
\usepackage{thm-restate}

\newcommand{\be}{\mathbf{e}}

\newcommand{\R}{\mathbb{R}}
\newcommand{\Z}{\mathbb{Z}}

\newcommand{\SL}{\mathrm{SL}}

\newtheorem{theorem}{Theorem}
\newtheorem{lemma}[theorem]{Lemma}
\newtheorem{proposition}[theorem]{Proposition}
\newtheorem{corollary}[theorem]{Corollary}

\newtheorem{question}[theorem]{Question}

\theoremstyle{definition}
\newtheorem{remark}[theorem]{Remark}

\makeatletter
\@namedef{subjclassname@2020}{\textup{2020} Mathematics Subject Classification}
\makeatother

\title{Densities of arithmetic Hecke triangle group orbits}

\author[Fairchild]{Samantha Fairchild}
\address[S. Fairchild]{Department of Mathematics and Computer Science\\Eindhoven university of Technology\\  PO Box 513
5600 MB \\Eindhoven\\The Netherlands}
\email{\href{mailto:s.k.fairchild@tue.nl}{\texttt{s.k.fairchild@tue.nl}}}
\urladdr{\href{https://sites.google.com/view/sfairchild/home}{sites.google.com/view/sfairchild}}

\author[Hanusa]{Christopher R.\ H.\ Hanusa}
\address[C.\ R.\ H.\ Hanusa]{Department of Mathematics \\ Queens College (CUNY) \\ 65-30 Kissena Blvd. \\ Queens, NY 11367\\ United States}
\email{\href{mailto:chanusa@qc.cuny.edu}{\texttt{chanusa@qc.cuny.edu}}}
\urladdr{\href{https://hanusa.xyz/}{hanusa.xyz}}



\begin{document}

\begin{abstract}
We give new proofs computing the asymptotic densities for orbits of the arithmetic Hecke triangle groups $\Gamma_q$ when $q=4$ and $q=6$. 
We use elementary number theory techniques along with basic properties of the Möbius function and Riemann zeta function with additional congruence conditions. 
This note actually came about by observing that, in the $q=4$ case, the underlying congruence condition partitions the set of coprime integer pairs into three classes of equal density.
\end{abstract}
\maketitle

\section{Introduction}
\label{sec:intro}
Let $\SL(2,\R)$ or $\SL(2,\Z)$ denote the space of $2\times 2$ matrices with determinant $1$ and entries in $\R$ or~$\Z$.
We consider the {\em Hecke triangle groups} $\Gamma_q < SL(2,\R)$ for integers $q\geq 3$ generated by 
$$\Gamma_q =  \left\langle \begin{pmatrix}0 & 1 \\ -1 & 0\end{pmatrix}, \begin{pmatrix}1 & \lambda_q \\ 0 & 1\end{pmatrix} \right\rangle,$$
where $\lambda_q = 2\cos(\pi/q)$ \cite{LangLangHecke}. The orbit $\Gamma_q\be_1$ is always a discrete subset of $\R^2$ \cite[Theorem~3.2]{Dalbo12}.  When $q=3$, $\Gamma_3 = \SL(2,\Z)$ so $\Gamma_3 \be_1 \subseteq \Z^2$. In fact, the determinant $1$ condition implies the orbit $\Gamma_3 \be_1$ is the set of \emph{primitive integers},
$$\Z^2_{prim} = \left\{\begin{pmatrix}
    a\\b
\end{pmatrix}\in \Z^2: \gcd(a,b) = 1\right\}.$$ While $\Gamma_3 \be_1$ is well understood, the purpose of this paper is to describe the distribution of $\Gamma_q \be_1$ over the Euclidean plane when $q=4$ and $q=6$, in which $\lambda_q=\sqrt{q/2}$. (See Figure~\ref{fig:TriangleGroupOrbitsq346}.)

\begin{figure}[ht]
    \centering
    \begin{subfigure}[b]{0.3\textwidth}
        \centering
        \includegraphics[width=\textwidth]{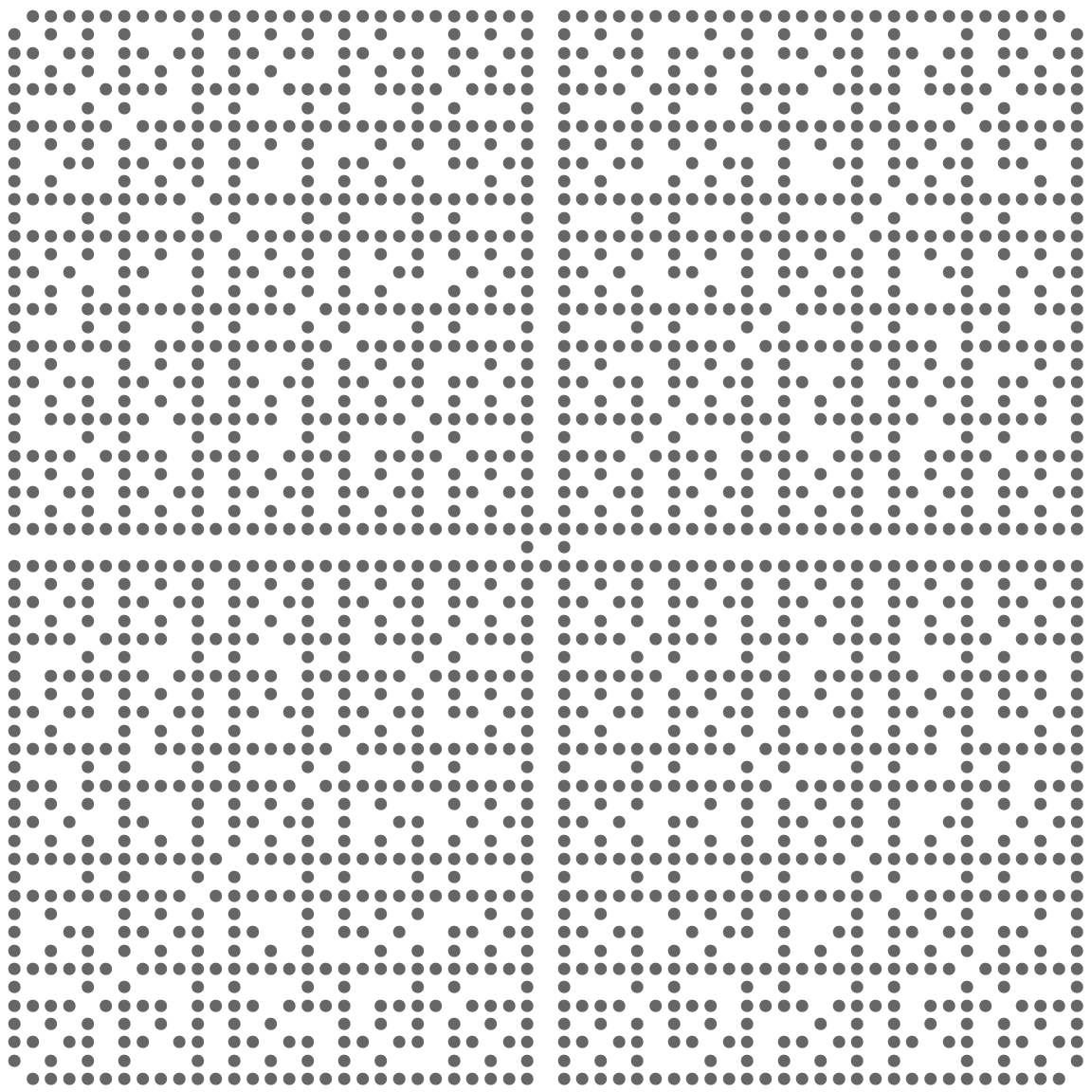}
        \caption{$q=3$}
        \label{fig:sub1}
    \end{subfigure}
    \hfill
    \begin{subfigure}[b]{0.3\textwidth}
        \centering
        \includegraphics[width=\textwidth]{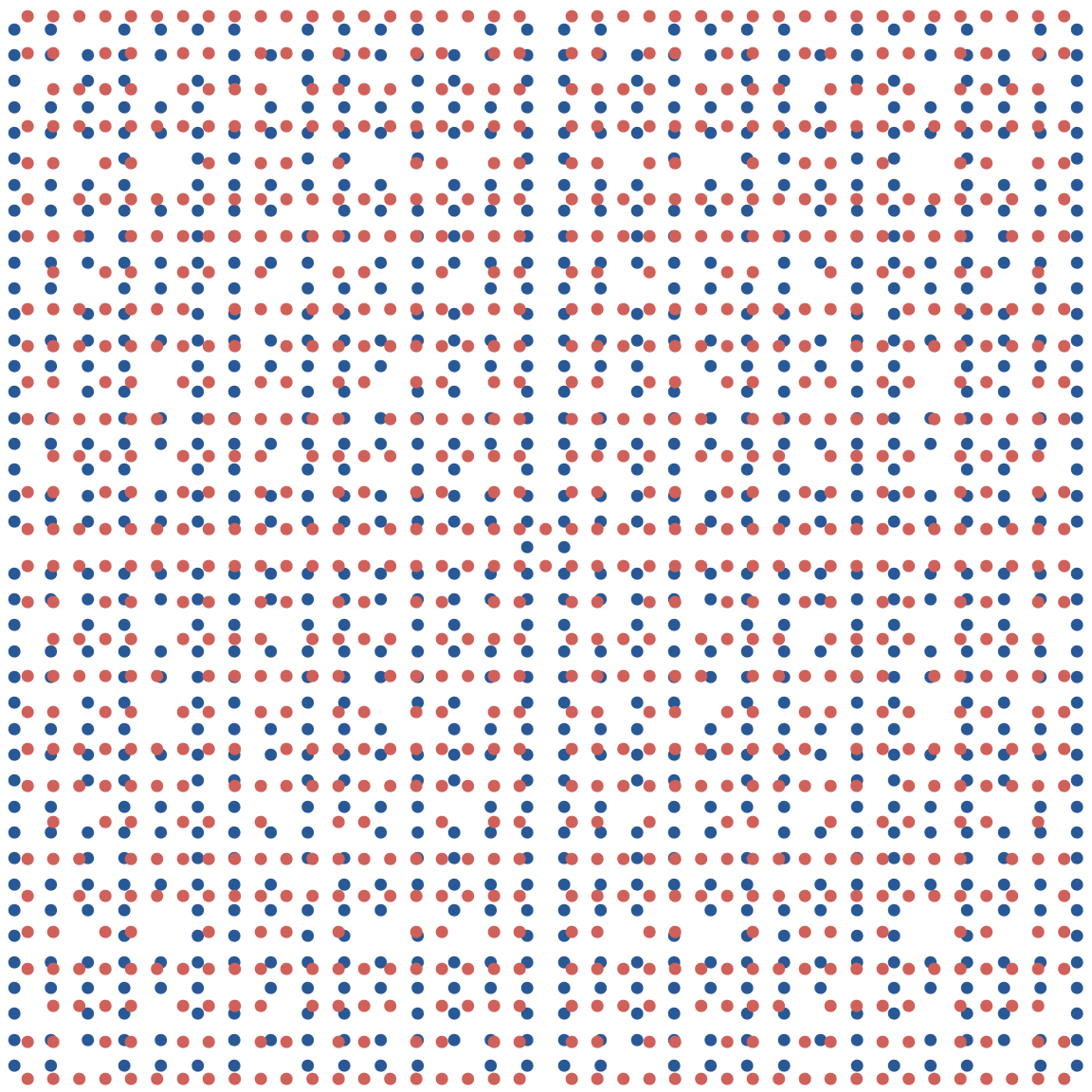}
        \caption{$q=4$}
        \label{fig:sub2}
    \end{subfigure}
    \hfill
    \begin{subfigure}[b]{0.3\textwidth}
        \centering
        \includegraphics[width=\textwidth]{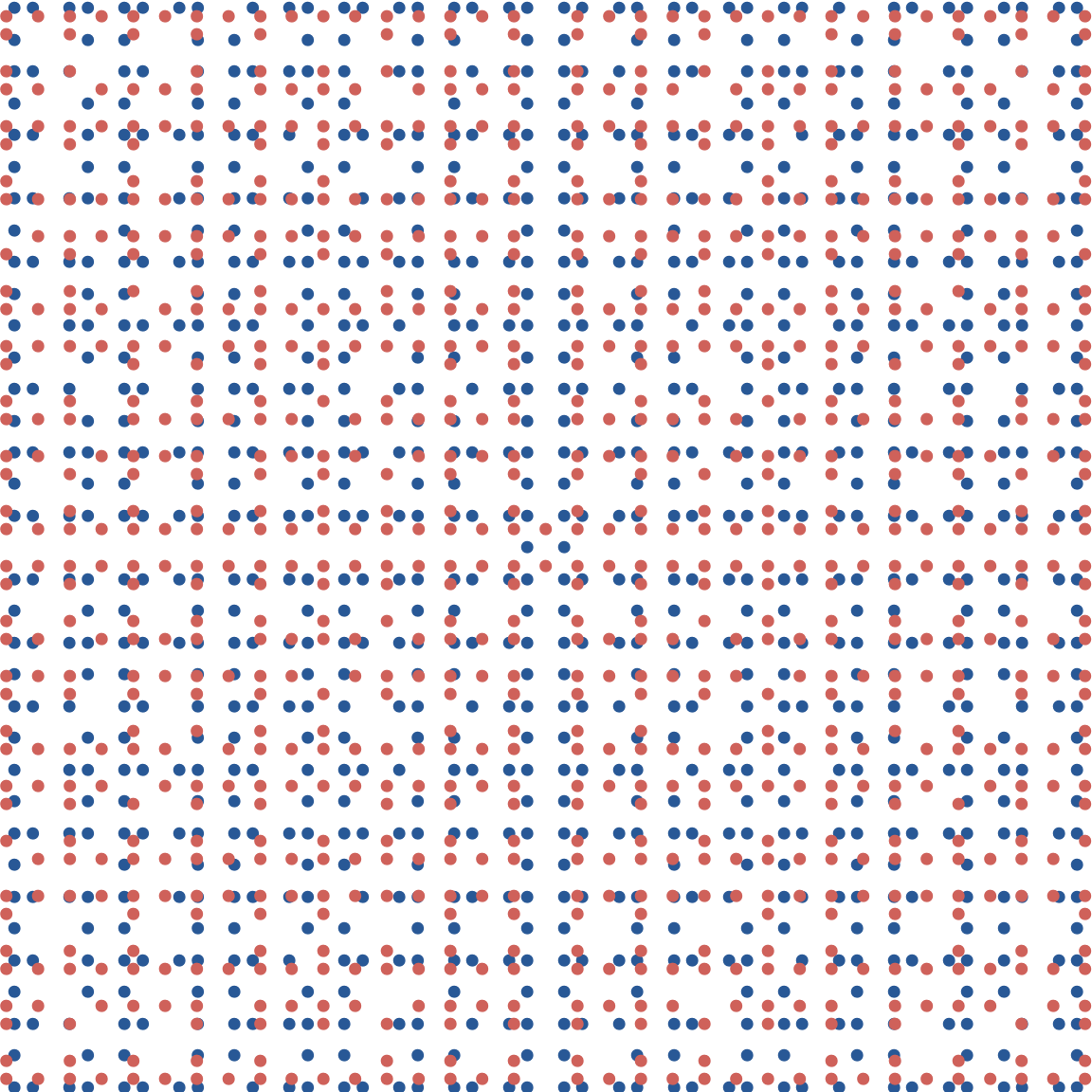}
        \caption{$q=6$}
        \label{fig:sub3}
    \end{subfigure}

    \caption{The orbits of $\Gamma_q\be_1$. For $q=3$ we obtain the primitive integer lattice. For $q=4$ and $q=6$ the points are of the form $(a,b\lambda_q)$ and $(a\lambda_q,b)$, which we indicate by dark blue and light red, respectively.}
    \label{fig:TriangleGroupOrbitsq346}
\end{figure}

The study of the distribution of the primitive integer lattice is a venerable problem going back to the time of Gauss. The \emph{primitive Gauss circle problem} is concerned with determining the error term when counting pairs of primitive integers in the disk $B(R)$ centered at the origin with radius~$R$, 
$$\#\{\Z^2_{prim} \cap B(R)\} \sim \frac{6}{\pi^2} \pi R^2,$$
where $f(R)\sim g(R)$ means $\lim_{R\to\infty} f(R)/g(R) = 1$.
Determining the error term is part of the study of the geometry of numbers; a proof of the asymptotic density, first given by Gauss, is seen for example in \cite[Section 4.1]{Olds}. While the error term of density depends on the shape of the region, the main term does not \cite[Theorem 3.9]{Apostol}. As such, by choosing the square $[-R,R]^2$ as our boundary, the asymptotic density is the same:
$$\lim_{R\to\infty}\frac{\#\{\Z^2_{prim} \cap B(R)\}}{\pi R^2} = \lim_{R\to\infty}\frac{\# \{\Z^2_{prim} \cap [-R,R]^2\}}{4 R^2} = \frac{6}{\pi^2}. $$
Changing the shape of the boundary allows us to prove density results.

Our focus is to understand what happens when discrete subsets of the plane are not necessarily part of a lattice, namely the subsets $\Gamma_q\be_1$ for $q= 4,6$.
The following theorem is our main result.

\begin{restatable}[]{theorem}{density}
\label{thm:density}
$$\lim_{R\to\infty}\frac{\# \{\Gamma_4\be_1\cap [-R,R]^2\}}{4R^2} = \frac{4\sqrt{2}}{\pi^2}
\qquad\textup{ and }\qquad
\lim_{R\to\infty}\frac{\# \{\Gamma_6\be_1\cap [-R,R]^2\}}{4R^2} = \frac{3\sqrt{3}}{\pi^2}.
$$
\end{restatable}

\begin{proof} Figure~\ref{fig:TriangleGroupOrbitsq346} provides the intuition for how we prove Theorem~\ref{thm:density}. When $q=4,6$, the points in $\Gamma_q\be_1$ partition nicely into two symmetric sets $C_q$ and $D_q$ given explicilty in \Cref{prop:description} below. Furthermore, Corollary~\ref{cor:CDdensity} at the end of \Cref{sec:primes} proves that 
$$\lim_{R\to\infty}\frac{\# \{C_q\cap [-R,R]^2\}}{4R^2} 
=
\lim_{R\to\infty}\frac{\# \{D_q\cap [-R,R]^2\}}{4R^2} = \frac{6}{\pi^2}\frac{\lambda_q}{\lambda_q^2+1},
$$
from which Theorem~\ref{thm:density} holds.
\end{proof}

\begin{proposition}
\label{prop:description}
Let $q=4,6$. The set $\Gamma_q\be_1$ partitions into the two sets
\[C_q=\left\{\begin{pmatrix}
      a\\ b\lambda_q
  \end{pmatrix}:\begin{pmatrix}a\\b\end{pmatrix}\in \Z^2_{prim}, a\not\equiv 0\bmod \lambda_q^2\right\}\textup{ and }D_q=\left\{\begin{pmatrix}
      a\lambda_q\\ b
  \end{pmatrix}:\begin{pmatrix}a\\b\end{pmatrix}\in \Z^2_{prim}, b\not\equiv 0\bmod \lambda_q^2\right\}.\]
\end{proposition}
\begin{proof}
       By \cite[Remark 3.9]{LangLangHecke}, columns of $\Gamma_q\be_1$ are of the form
    $(a,b\lambda_q)$ with $\gcd(a , b\lambda_q^2) = 1$ or $(a\lambda_q,b)$ with $\gcd(a \lambda_q^2, b) = 1$ for $a,b \in \Z$. We define $C_q$ to be the set of vectors of the first form and $D_q$ to be the set of vectors of the second form.
    Since $\lambda_q^2$ is prime, $\gcd(a , b\lambda_q^2) = 1$ is equivalent to the conditions that $\gcd(a,b)=1$ and $a\not\equiv 0\bmod \lambda_q^2$, which establishes our characterization of $C_q$. Our characterization of $D_q$ follows by symmetry.
\end{proof}

This work features congruences modulo the prime numbers $\lambda_q^2=q/2=2$ and $3$. In \Cref{sec:primes}, the results to obtain \Cref{cor:CDdensity} are proven in full generality for all prime numbers. This leads to \Cref{cor:prime_density}, an asymptotic density result about which primitive pairs occur in both $C_q$ and $D_q$. Finally, \Cref{sec:conclusion} concludes with some questions for future study.

\medskip
We now discuss connections to previous work. The groups $\Gamma_q$ for $q=3,4,6$ are the only arithmetic Hecke triangle groups \cite{Arithmetic_Triangle_Groups_Takeuchi}. Studying their number-theoretic properties is an active area of research, for example, in extreme value theory for $q=3$ \cite{Pollicott_EVT},  Diophantine approximation for $q=4$ \cite{H4_Diophantine}, and spectral theory for $q=6$ \cite{H6_Spectrum}. 

Furthermore, each $\Gamma_q$ forms a discrete subgroup of $\SL(2,\R)$, which is a \emph{nonuniform lattice}. 
Nonuniform lattice orbits were the primary objects of study in the  papers \cite{Veech89, Veech98} motivated by the study of translation surfaces (see \cite{Translation_Surfaces_Book}). The state of the art for asymptotic counting estimates on nonuniform lattice orbits uses analytic number theory techniques \cite{BurrinNevoRuhrWeiss, BurrinFairchildChaika}. In both of these situations, the main leading term is obtained through computing the co-volume. In contrast, in this article we rely on the density of the primitive integers and use scaling factors.

\section{Computing Densities}\label{sec:primes}

In this section, we prove the asymptotic density of the sets $C_q$ and $D_q$ from \Cref{prop:description}. To do so, we note that for $q= 4,6$, $p = q/2$ is a prime number. We will prove results featuring congruences modulo any prime number, and explore the scaled sets 
\[C_p'=\left\{\begin{pmatrix}a\\b\end{pmatrix}\in \Z^2_{prim}: a\not\equiv 0\bmod p\right\}\!\!\textup{ and }D_p'=\left\{\begin{pmatrix}a\\b\end{pmatrix}\in \Z^2_{prim}: b\not\equiv 0\bmod p\right\}\!\,.\]
Figure~\ref{fig:posprim} provides a visualization of these sets along with their intersection 
\[E_p'=\left\{\begin{pmatrix}a\\b\end{pmatrix}\in \Z^2_{prim}: a,b\not\equiv 0\bmod p\right\}.\]
The blue circles with vertical dashes are the elements of $C_p'\setminus E_p'$, the red circles with horizontal dashes are the elements of $D_p'\setminus E_p'$, and the purple circles with crosses are the elements of $E_p'$. 

We want to understand the asymptotic densities of subsets of primitive integers subject to modular constraints, which motivates the following. 

\begin{figure}[b]
    \centering
    \begin{subfigure}[b]{0.3\textwidth}
        \centering
    \includegraphics[width=\textwidth]{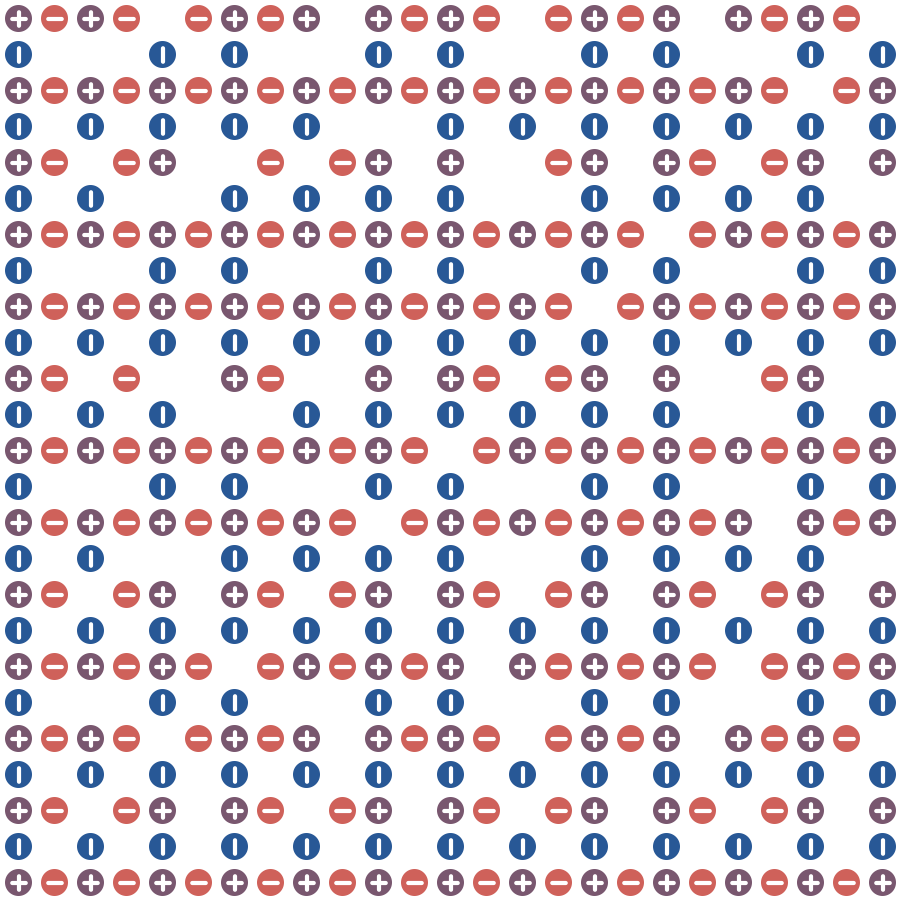}
        \caption{$ p=2$}
        \label{fig:q4}
    \end{subfigure}
    \qquad\qquad
    \begin{subfigure}[b]{0.3\textwidth}
        \centering
        \includegraphics[width=\textwidth]{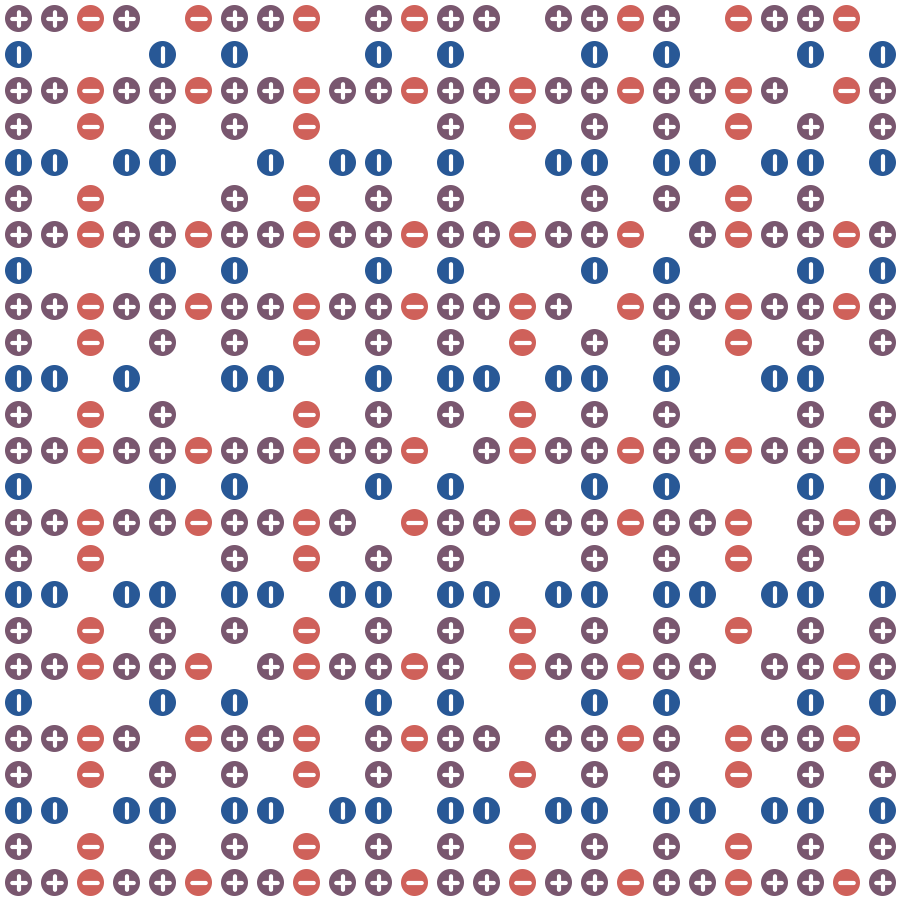}
        \caption{$p=3$}
        \label{fig:q6}
    \end{subfigure}
    \caption{The positive primitive pairs $(a,b)$, shaded dark blue with vertical dashes when $(a,b)\in C_p'$, shaded light red with horizontal dashes when $(a,b)\in D_p'$, or shaded purple with crosses when $(a,b)$ is in $C_p' \cap D_p'$.} 
    \label{fig:posprim}
\end{figure}

    \begin{lemma}\label{lem:counting_rectangles}
        Let $\alpha, \beta \in (0,1]$. Let $p$ be prime and let $a_0$ be an integer $1\leq a_0<p$
    \begin{equation}\label{eq:bigeq2}
    \lim_{R\to\infty}\frac{1}{R^2}\cdot
    \#\bigg\{
        \!\!\begin{tabular}{c}
              $(a,b) \in \Z^2_{prim}\cap [0,\alpha R]\times [0, \beta R]$ such that\\
               $a\equiv a_0\bmod p$ and $b\not\equiv 0 \bmod{p}$
        \end{tabular}\!\! \bigg\} =  \frac{6}{\pi^2}\cdot \frac{\alpha \beta}{p+1}
        \end{equation}
    \end{lemma}
    \begin{proof}
            Set 
            \[P(R) =\#\bigg\{
        \!\!\begin{tabular}{c}
              $(a,b) \in \Z^2_{prim}\cap [0,\alpha R]\times [0, \beta R]$ such that\\
               $a\equiv a_0\bmod p$ and $b\not\equiv 0 \bmod{p}$
        \end{tabular}\!\! \bigg\}. \]
        Recall the M{\"o}bius function satisfies $\sum_{d|n} \mu(d) = 1$ exactly when $n=1$, and otherwise is zero. Therefore, \[\sum_{\substack{d\in \Z\\ d|a, d|b}} \mu(d)=\left.\begin{cases}
        1 & \textup{if $\gcd(a,b)=1$} \\ 
        0 & \textup{if $\gcd(a,b)>1$}
        \end{cases}\right\},\]        
        which allows us to write
        \[P(R) = \sum_{\substack{1\leq a\leq \alpha R \\1\leq b\leq \beta R\\ a \equiv a_0 \bmod{p}\\ b \not\equiv 0 \bmod{p}}} \sum_{\substack{d\in \Z\\ d|a, d|b}} \mu(d).\]

        We now want to interchange the order of the finite sum. Rewrite $a = dx$ and $b = dy$ for integers $x,y$, where the bounds become $0\leq x \leq \frac{\alpha R}{d}$ and $0\leq y \leq \frac{\beta R}{d}.$ Since neither $a$ nor $b$ is divisible by~$p$, this implies each $d$ contributing to the sum satisfies $\gcd(d,p) = 1$. Moreover this implies $d$ is invertible modulo $p$ giving the congruence conditions $x \equiv a_0 d^{-1} \bmod{p}$ and $y \not\equiv 0 \bmod{p}.$ Thus, we have rewritten the sum to be
        \begin{equation*} \label{eq:reordering} P(R) = \sum_{\substack{d\leq \min\left\{\alpha R, \beta R\right\}\\ \gcd(d,p) = 1}} \mu(d) \cdot \#\left\{1\leq x \leq \frac{\alpha R}{d}: x \equiv a_0 d^{-1} \bmod {p}\right\} \cdot\#\left\{1\leq y \leq \frac{\beta R}{d}: y\not\equiv 0 \bmod{p} \right\}.\end{equation*}

        The number of integers with a fixed congruence class modulo $p$ is exactly $1/p$ of the integers, up to an error term of whether or not $\frac{\alpha R}{d}$ is an integer. Thus,
        $$\#\{1\leq x \leq \frac{\alpha R}{d}: x \equiv a_0 d^{-1} \bmod {p}\} = \frac{\alpha R}{p d} + O(1)$$
        and
        $$\#\{1\leq y \leq \frac{\beta R}{d}: y\not\equiv 0 \bmod{p} \} = \frac{\beta R(p-1)}{d p} + O(1).$$
        As a consequence,
        \begin{equation}\label{eq:sum_to_compute} P(R) = \frac{\alpha \beta (p-1)}{p^2} R^2\sum_{\substack{d\leq \min\{\alpha R, \beta R\} \\ \gcd(d,p) = 1}} \frac{\mu(d)}{d^2}  +O\left(\frac{R}{d} + 1\right).\end{equation}

        To compute the sum, recall the identity with the Riemann $\zeta$ function
        \[\sum_{d\geq 1} \frac{\mu(d)}{d^2} = \prod_{t \text{ prime}} \left(1-\frac{1}{t^2}\right) = \frac{1}{\zeta(2)} = \frac{6}{\pi^2}. \]
        The coprime condition removes the prime $p$ so the sum evaluates to
        \begin{equation}\label{eq:summationmu}\sum_{\substack{d\geq 1\\\gcd(d,p) = 1}} \frac{\mu(d)}{d^2} = \frac{6}{\pi^2}\frac{p^2}{p^2-1}. \end{equation}
        Since the sum in Equation~\eqref{eq:sum_to_compute} converges, for fixed $R$ the value in Equation~\eqref{eq:summationmu} is the leading term with error $O(1)$.
        Moreover $\sum_{d\leq R} \frac{1}{d} =O( \log(R))$ and $\sum_{d\leq R} 1 =O(R)$, so
        \begin{equation*} P(R) = \frac{6}{\pi^2} \frac{\alpha \beta}{(p+1)}R^2 + O(R\log(R)),\end{equation*}
        from which Equation~\eqref{eq:bigeq2} follows.
        \end{proof}

\Cref{lem:counting_rectangles} allows us to compute formulas for the asymptotic densities for $C'_p$ and $D'_p$ but also their intersection $E'_p$. In these calculations we now expand our domains to include negative values of $a$ and $b$.

\begin{corollary}\label{cor:prime_density}
    Let $p$ be prime. The asymptotic densities for $C_p'$ and $D_p'$ are 
\[\lim_{R\to\infty}\frac{\#\{C'_{p}\cap ([-\alpha R,\alpha R]\times [-\beta R,\beta R])\}}{(2\alpha R)(2\beta R)} = \lim_{R\to\infty}\frac{\#\{D_{p}'\cap ([-\alpha R,\alpha R]\times [-\beta R,\beta R])\}}{(2\alpha R)(2\beta R)} = \frac{6}{\pi^2} \frac{p}{p+1}\]
and the asymptotic density for $E_p'$ is 

\[\lim_{R\to\infty}\frac{\#\{E'_{p}\cap ([-\alpha R,\alpha R]\times [-\beta R,\beta R])\}}{(2\alpha R)(2\beta R)} = \frac{6}{\pi^2} \frac{p-1}{p+1}.\]
\end{corollary}
\begin{proof}
Apply \Cref{lem:counting_rectangles} and sum over equal contributions from the $p-1$ residue classes $1\leq a_0\leq p-1$ to arrive at the above asymptotic density formula for $E'_p$. Noting the symmetry between $C'_p$ and $D'_p$ and that the asymptotic density of $\Z^2_{prim}$ is $6/\pi^2$ implies that the asymptotic densities for $C'_p\setminus E'_p$ and $D'_p\setminus E'_p$ are both $\frac{6}{\pi^2}\frac{1}{(p+1)}$, from which the asymptotic densities for $C'_p$ and $D'_p$ follow. 
\end{proof}

\begin{remark}
\label{rem:equal}
\Cref{cor:prime_density} proves that when $p=2$, the asymptotic densities of the blue, red, and purple sets in the square in \Cref{fig:posprim}(a) are the same; they each equal $2/\pi^2$.
\end{remark}

\begin{corollary}\label{cor:CDdensity}
Let $p=2,3$ and $q=2p$. The asymptotic densities for $C_{q}$ and $D_{q}$ are 
    \[\lim_{R\to\infty}
     \frac{
     \#\{C_{q}\cap ([-R,R]^2)\}
     }{4R^2} 
     = 
     \lim_{R\to\infty}
     \frac{
     \#\{D_{q}\cap ([-R,R]^2)\}
     }{4R^2} 
     = \frac{1}{\sqrt{p}}\frac{6}{\pi^2}\frac{p}{p+1}.\]
\end{corollary}
\begin{proof}
Apply Corollary~\ref{cor:prime_density}. For $C_p'$, use $\alpha=1$ and $\beta=1/\sqrt{p}$; for $D_p'$ use $\alpha=1/\sqrt{p}$ and $\beta=1$ so
    \[\lim_{R\to\infty}
     \frac{
     \#\big\{C'_p\cap 
     \big([-R,R]\times[-\frac{1}{\sqrt{p}}R,\frac{1}{\sqrt{p}}R]\big)\big\}
     }{(2R)\big(\frac{2}{\sqrt{p}}R\big)} 
     = 
     \lim_{R\to\infty}
     \frac{
     \#\big\{D'_p\cap 
     \big([-\frac{1}{\sqrt{p}}R,\frac{1}{\sqrt{p}}R]\times[-R,R]\big)\big\}
     }{\big(\frac{2}{\sqrt{p}}R\big)(2R)} 
     = \frac{6}{\pi^2} \frac{p}{p+1}.\]
    Scaling the points in $C'_p$ vertically and $D'_p$ horizontally by a factor of $\sqrt{p}$ does not change the number of points contributing to the numerator. Thus,
    \[\lim_{R\to\infty}
     \frac{
     \#\{C_{q}\cap ([-R,R]^2)\}
     }{(2R)\big(\frac{2}{\sqrt{p}}R\big)} 
     = 
     \lim_{R\to\infty}
     \frac{
     \#\{D_{q}\cap ([-R,R]^2)\}
     }{\big(\frac{2}{\sqrt{p}}R\big)(2R)} 
     = \frac{6}{\pi^2} \frac{p}{p+1}.\]
     Corollary~\ref{cor:CDdensity} follows directly.
\end{proof}

\section{Concluding remarks}
\label{sec:conclusion}
This work originated through an investigation into the structure of the visualizations in \Cref{fig:TriangleGroupOrbitsq346}. After scaling the sets $C_q$ and $D_q$ and noticing the partition of the primitive pairs into three disjoint sets, we were surprised to learn that these sets were asymptotically equinumerous. (See \Cref{rem:equal}.) 

This leads us to ask whether other partitions of the primitive integers have structure of interest. 

\begin{question}
What can be said about subsets or partitions of $\Z^d_{prim}$ that are not related to congruences modulo  $p$? 
\end{question}

Our results in \Cref{sec:primes} for $p=2,3$ apply to the Hecke triangle groups, although they were proven in full generality for all primes $p$. 

\begin{question}
What applications are there for \Cref{cor:prime_density} for $p\ne 2,3$?
\end{question}

Finally, we wonder what other methods might work for other arithmetic non-uniform lattices.

\begin{question}
Is there a class of arithmetic non-uniform lattices that lends itself to similarly elementary methods?
\end{question}

\section*{Acknowledgments}

We gratefully acknowledge the support of the Institut Henri Poincar\'e (UAR 839 CNRS-Sorbonne Université), and LabEx CARMIN (ANR-10-LABX-59-01). This project was started through a collaboration during the {\em Illustration as a Mathematical Research Technique} trimester in January~2026.  SF was partially supported by the project Computational Connections in Algebra and Geometry with file number VI.Veni.242.089 of the research programme Veni ENW which is (partly) financed by the Dutch Research Council (NWO) under the grant ID \url{https://doi.org/10.61686/VYNXE24650}. 

\section*{AI Statement}
AI was used in developing various proofs; however, all words, exposition, and figures are solely the authors' contributions. 
\bibliographystyle{alpha}
\bibliography{Sources}
\end{document}